\documentclass[11pt,reqno]{article}
\usepackage{amssymb,amscd,amsmath,amsthm,color}
\newtheorem{thm}{Theorem}[section]

\newtheorem{lemma}[thm]{Lemma}
\newtheorem{cor}[thm]{Corollary}

\newtheorem{remark}[thm]{Remark}

\numberwithin{equation}{section}

\def\cH{\mathcal{H}}
\def\bC{\mathbb{C}}
\def\bN{\mathbb{N}}
\def\bM{\mathbb{M}}

\def\bR{\mathbb{R}}
\def\Re{\mathrm{Re}\,}
\def\Im{\mathrm{Im}\,}
\def\eps{\varepsilon}

\begin{document}
\baselineskip=16pt
\allowdisplaybreaks

\ \vskip 1cm 
\centerline{\LARGE Operator $k$-tone functions and analytic functional calculus}
\bigskip
\bigskip
\centerline{\Large
Fumio Hiai\footnote{Supported in part by Grant-in-Aid for Scientific Research (C)21540208. \\
\quad\ \ {\it E-mail address:} hiai.fumio@gmail.com}}

\medskip
\begin{center}
$^1$\,Tohoku University (Emeritus), \\
Hakusan 3-8-16-303, Abiko 270-1154, Japan
\end{center}

\medskip
\begin{abstract}
Operator $k$-tone functions on an open interval of the real line, which are higher order extensions
of operator monotone and convex functions, are characterized via certain inequalities for the real
and imaginary parts of analytic functional calculus by those functions.



\end{abstract}

\section{Introduction}

Theory of operator/matrix monotone and convex function initiated by L\"owner \cite{Lw} and Kraus
\cite{Kr} has been a significant topic both theoretically and in applications in Hilbert space
operator theory and matrix analysis. In \cite{FHR} we introduced the notion of operator/matrix
$k$-tone functions, which is a higher order extension of operator/matrix monotone and convex
functions, and obtained several characterizations and integral representations of those functions.
An operator $k$-tone function $f$ on an open interval $(a,b)$ is automatically analytic and is
characterized, among others, in terms of positivity of the $k$th derivative of functional calculus
for $f$ in such a way that
$$
D^kf(A;B):={d^k\over dt^k}\,f(A+tB)\Big|_{t=0}\ge0
$$
for bounded self-adjoint operators $A,B$ such that the spectrum $\sigma(A)$ is in $(a,b)$ and $B\ge0$.
Moreover, this property enables us to compare $f(A+B)$ with its Taylor form as
\begin{equation}\label{F-1.1}
f(A+B)\ge\sum_{m=0}^{k-1}{1\over m!}\,D^mf(A;B)
\end{equation}
whenever the operator norm $\|B\|$ is small enough so that $\sigma(A+B)$ is in $(a,b)$.

In the present paper we characterize operator $k$-tone functions in terms of analytic functional
calculus of non-self-adjoint operators, instead of functional calculus of self-adjoint operators as
in \eqref{F-1.1}. We consider a continuous real function $f$ on $(a,b)$ which is analytically
continued to the upper half-plane $\Im z>0$. For an operator $X=A+iB$ with self-adjoint $A,B$, if
either $\sigma(A)\subset(a,b)$ or $B>0$ (i.e., $B$ is positive and invertible), then we can define
the analytic functional calculus $f(X)$. Our main theorem says that the condition of $f$ being
operator $k$-tone is characterized by inequalities between the real part $\Re f(X)$ and a certain
Taylor form of even powers, or between the imaginary part $\Im f(X)$ and a Taylor form of odd powers.
A remarkable point here is that our characterizing inequalities are of four different kinds with
inequality signs and even/odd powers, depending on $k$ (mod $4$). For instance, the operator
$4k$-tonicity of $f$ is characterized by the inequality
\begin{equation}\label{F-1.2}
\Re f(A+iB)\ge\sum_{m=0}^{2k-1}{(-1)^m\over(2m)!}D^{2m}f(A;B)
\end{equation}
for every self-adjoint operators $A,B$ with $\sigma(A)\subset(a,b)$.

The paper is organized as follows. Section 2 contains brief accounts on analytic functional calculus
and operator $k$-tone functions. In Section 3 we present the main theorem characterizing operator
$k$-tone functions in terms of analytic functional calculus as in \eqref{F-1.2}. Finally in Section 4
we restrict our considerations to operator monotone and convex functions, and obtain further
characterization results.

\section{Preliminaries}

\subsection{Notations}

In this paper $\cH$ is a separable infinite-dimensional Hilbert space and $B(\cH)$ is the set of all
bounded linear operators on $\cH$. We denote by $B(\cH)^{sa}$ the set of all self-adjoint
$A\in B(\cH)$ and $B(\cH)^+$ the set of all positive $A\in B(\cH)^{sa}$. In addition, $\bM_n$ is the
$n\times n$ complex matrix algebra (i.e., $\bM_n=B(\bC^n)$ on the $n$-dimensional Hilbert space
$\bC^n$), $\bM_n^{sa}$ the $n\times n$ Hermitian matrices, and $\bM_n^+$ the $n\times n$ positive
semidefinite matrices. For $A\in B(\cH)^{sa}$ (also $A\in\bM_n^{sa}$) we write $A>0$ when $A$ is
positive and invertible. For any open interval $(a,b)$ of $\bR$ we denote by $B(\cH)^{sa}(a,b)$ the
set of all $A\in B(\cH)^{sa}$ whose spectrum $\sigma(A)$ is in $(a,b)$, and similarly by
$\bM_n^{sa}(a,b)$ the set of all $A\in\bM_n^{sa}$ with eigenvalues in $(a,b)$.

\subsection{Analytic functional calculus}

Let $f$ be a continuous real function on an interval $(a,b)$, where $-\infty\le a<b\le\infty$, and
assume that $f$ is analytically continued to the upper half-plane $\bC^+$, that is, there is a
continuous function $\tilde f(z)$, $z\in\bC^+\cup(a,b)$, such that $\tilde f$ is analytic in $\bC^+$
and $\tilde f(x)=f(x)$ for all $x\in(a,b)$. For simplicity we denote the extension $\tilde f$ by the
same $f$. By reflection principle, $f$ is further extended to the lower half-plane $\bC^-$ in such a
way that $f(\bar z)=\overline{f(z)}$, $z\in\bC^-$, so that $f$ is analytic in
$(\bC\setminus\bR)\cup(a,b)$.

Let $X\in B(\cH)$ or, in particular, $X\in\bM_n$ with $\sigma(X)\subset(\bC\setminus\bR)\cup(a,b)$.
For the above $f$ one can define the analytic functional calculus of $X$ as
$$
f(X):={1\over2\pi i}\int_\Gamma f(\zeta)(\zeta I-X)^{-1}\,d\zeta,
$$
where $\Gamma$ is a piecewise smooth closed curve in $(\bC\setminus\bR)\cup(a,b)$
surrounding $\sigma(X)$. Furthermore, for any $Z\in B(\cH)$ and any $z\in\bC$ with
sufficiently small $|z|$ so that $\sigma(X+zZ)$ is inside $\Gamma$, the analytic functional
calculus $f(X+zZ)$ is also defined as
$$
f(X+zZ):={1\over2\pi i}\int_\Gamma f(\zeta)(\zeta I-X-zZ)^{-1}\,d\zeta,
$$
and the $B(\cH)$-valued analytic function $z\mapsto f(X+zZ)$ admits the Taylor expansion
\begin{equation}\label{F-2.1}
f(X+zZ)=\sum_{m=0}^\infty{z^m\over m!}D^mf(X;Z),
\end{equation}
where
\begin{align}\label{F-2.2}
D^mf(X;Z)&:={d^m\over dz^m}f(X+zZ)\Big|_{z=0} \nonumber\\
&\ ={m!\over2\pi i}\int_\Gamma f(\zeta)((\zeta I-X)^{-1}Z)^m(\zeta I-X)^{-1}\,d\zeta,
\qquad m\ge0.
\end{align}
Furthermore, when $A\in\bM_n^{sa}(a,b)$, for any $m\in\bN$ one has the $m$th {\it Fr\'echet
derivative} $D^mf(A)$, an $m$-multilinear map of $\bM_N^{sa}\times\cdots\times\bM_n^{sa}$ ($m$ times)
to $\bM_n^{sa}$, of the usual functional calculus $A\in\bM_n^{sa}(a,b)\mapsto f(A)\in\bM_n^{sa}$
(see, e.g., \cite[\S2.3]{Hi}). Then it is well-known that
$$
D^mf(A;B)=D^mf(A)(\underbrace{B,\dots,B}_{m}),\qquad B\in\bM_n^{sa}.
$$

\begin{lemma}\label{L-2.1}
Let $X\in B(\cH)$.
\begin{itemize}
\item[\rm(1)] If $\Im X>0$, then $\sigma(X)\subset\bC^+$,
\item[\rm(2)] If $\sigma(\Re X)\subset(a,b)$, then $\sigma(X)\subset(\bC\setminus\bR)\cup(a,b)$.
\end{itemize}
\end{lemma}

\begin{proof}
(1) was given in \cite{Ep}. In fact, if $X=A+iB$ with $A,B\in B(\cH)^{sa}$ and $B>0$, then the
inverse of $X$ is
$$
X^{-1}=B^{-1/2}(B^{-1/2}AB^{-1/2}+iI)^{-1}B^{-1/2}.
$$
This implies that $X-zI$ is invertible for every $z\in\bC$ with $\Im z\le0$.

(2)\enspace
Assume that $X=A+iB$ where $A,B\in B(\cH)^{sa}$ with $aI<A<bI$. Then, for every $\lambda\in\bR$ with
$\lambda<a$, we have $A-\lambda I>0$ and
$$
X-\lambda I=(A-\lambda I)^{1/2}\{I+i(A-\lambda I)^{-1/2}B(A-\lambda I)^{-1/2}\}
(A-\lambda I)^{1/2}
$$
is invertible. Also, for every $\lambda\in\bR$ with $\lambda>b$, we have $\lambda I-A>0$ and
$$
\lambda I-A=(\lambda I-A)^{1/2}\{I-i(\lambda I-A)^{-1/2}B(\lambda I-A)^{-1/2}\}
(\lambda I-A)^{1/2}
$$
is invertible. Hence $\sigma(X)\subset(\bC\setminus\bR)\cup(a,b)$.
\end{proof}

Therefore, we can define the analytic functional calculus $f(X)$ in the cases (1) and (2) of the
above lemma.

\subsection{Operator $k$-tone functions}

A real function $f$ on $(a,b)$ is said to be {\it operator monotone} if $A\le B$ implies
$f(A)\le f(B)$ for every $A,B\in B(\cH)^{sa}(a,b)$, and {\it operator monotone decreasing} if $-f$
is operator monotone. Also, $f$ is said to be {\it operator convex} if 
$$
f(\lambda A+(1-\lambda)B)\le\lambda f(A)+(1-\lambda)f(B),\qquad0<\lambda<1,
$$
for all $A,B\in B(\cH)^{sa}(a,b)$, and {\it operator concave} if $-f$ is operator convex. The theory
of operator monotone and operator convex functions was initiated by L\"owner \cite{Lw} and
Kraus \cite{Kr}, which was further developed in \cite{HP} and others. For details, see, e.g.,
\cite[\S V.4]{Bh1} and also \cite{An1,Do,Hi}.

In \cite{FHR} we introduced the notion of {\it operator $k$-tone} functions for $k\in\bN$, extending
operator monotone functions (when $k=1$) and operator convex functions (when $k=2$). Since the
definition itself of operator $k$-tone functions in \cite[Definition 4.1]{FHR} is a bit long to state
in detail, we here recall its equivalent conditions for convenience. For a real continuous function
on $(a,b)$, among several equivalent conditions of $f$ being operator $k$-tone are the following:

\begin{itemize}
\item[(a)] $f$ is $C^k$ on $(a,b)$ and
$$
{d^k\over dt^k}\,f(A+tB)\Big|_{t=0}\ge0
$$
for every $A\in\bM_n^{sa}(a,b)$, $B\in\bM_n^+$, and for every $n\in\bN$.
\item[(b)] $f$ is analytic on $(a,b)$ and
$$
{d^k\over dt^k}\,f(A+tB)\Big|_{t=0}\ge0
$$
for every $A\in B(\cH)^{sa}(a,b)$ and every $B\in B(\cH)^+$, where the above derivative of order $k$
is well defined in the operator norm.
\item[(c)] $f$ is $C^k$ on $(a,b)$ and
$$
f(A+B)\ge\sum_{m=0}^{k-1}{1\over m!}D^mf(A)(\underbrace{B,\dots,B}_m)
$$
for every $A\in\bM_n^{sa}(a,b)$ and $B\in\bM_n^+$ such that $A+B\in\bM_n^{sa}(a,b)$, and for every
$n\in\bN$. 
\end{itemize}

The above conditions (a)--(c) equivalent to operator $k$-tonicity are found in
\cite[Proposition 2.6, Theorem 3.3]{FHR}. In the above conditions, if $k$ is even, then $X$ can be a
general matrix in $\bM_n^{sa}$ or operator in $B(\cH)^{sa}$.

Furthermore, it is seen from \cite[Corollary 3.4]{FHR} that an operator $k$-tone function on $(a,b)$
is analytically continued to $\bC^+\cup\bC^-$ so that the extended $f$ is analytic in
$(\bC\setminus\bR)\cup(a,b)$, as an operator monotone function on $(a,b)$ does so.

\section{Characterization of operator $k$-tone functions via analytic functional calculus}

The aim of this section is to prove the next theorem, which is our main result of the paper
characterizing operator $k$-tone functions on $(a,b)$ via analytic functional calculus. The
statements are divided into four cases depending on $k$ (mod~$4$).

\begin{thm}\label{T-3.1}
Let $f$ be a continuous real function on $(a,b)$, $-\infty\le a<b\le\infty$, analytically continued
to $\bC^+$. Then for each $k\in\bN$ the following assertions hold:
\begin{itemize}
\item[\rm(1)] $f$ is operator $(4k-2)$-tone if and only if
$$
\Re f(A+iB)\le\sum_{m=0}^{2k-2}{(-1)^m\over(2m)!}D^{2m}f(A;B)
$$
for every $A\in B(\cH)^{sa}(a,b)$ and every $B\in B(\cH)^{sa}$, or equivalently, for every
$A\in\bM_n^{sa}(a,b)$ and every $B\in\bM_n^+$ of any $n\in\bN$.
\item[\rm(2)] $f$ is operator $4k$-tone if and only if
$$
\Re f(A+iB)\ge\sum_{m=0}^{2k-1}{(-1)^m\over(2m)!}D^{2m}f(A;B)
$$
for every $A\in B(\cH)^{sa}(a,b)$ and every $B\in B(\cH)^{sa}$, or equivalently, for every
$A\in\bM_n^{sa}(a,b)$ and every $B\in\bM_n^+$ of any $n\in\bN$.
\item[\rm(3)] $f$ is operator $(4k-3)$-tone if and only if
$$
\Im f(A+iB)\ge\sum_{m=1}^{2k-2}
{(-1)^{m-1}\over(2m-1)!}D^{2m-1}f(A;B)
$$
for every $A\in B(\cH)^{sa}(a,b)$ and every $B\in B(\cH)^+$, or equivalently, for every
$A\in\bM_n^{sa}(a,b)$ and every $B\in\bM_n^+$ of any $n\in\bN$. (The above right-hand side
is $0$ if $k=1$.)
\item[\rm(4)] $f$ is operator $(4k-1)$-tone if and only if
$$
\Im f(A+iB)\le\sum_{m=1}^{2k-1}
{(-1)^{m-1}\over(2m-1)!}D^{2m-1}f(A;B)
$$
for every $A\in B(\cH)^{sa}(a,b)$ and every $B\in B(\cH)^+$, or equivalently, for every
$A\in\bM_n^{sa}(a,b)$ and every $B\in\bM_n^+$ of any $n\in\bN$.
\end{itemize}
\end{thm}

\begin{proof}
As mentioned at the beginning of Section 2.2, the assumption on $f$ says that $f$ becomes an analytic
function in $(\bC\setminus\bR)\cup(a,b)$. Hence by Lemma \ref{L-2.1} the analytic functional calculus
$f(A+iB)$ can be defined whenever $A\in B(\cH)^{sa}(a,b)$ and $B\in B(\cH)^{sa}$.

Let us first prove the ``if\," part. Let $n\in\bN$ be arbitrary, and let $A\in\bM_n^{sa}(a,b)$ and
$B\in\bM_n^+$. By the Taylor expansion \eqref{F-2.1} with \eqref{F-2.2}, for every $l\in\bN$ we have
$$
f(A+i\eps B)=\sum_{m=0}^l{(i\eps)^m\over m!}\,D^mf(A;B)+O(\eps^{l+1})
\quad\mbox{as $\eps\searrow0$}.
$$
Since $D^mf(A;B)\in\bM_n^{sa}$ for all $m\ge0$, this implies that, for every $l\in\bN$,
\begin{align}
\Re f(A+i\eps B)&=\sum_{m=0}^l{(-1)^m\eps^{2m}\over(2m)!}\,D^{2m}f(A;B)+O(\eps^{2l+2})
\quad\mbox{as $\eps\searrow0$}, \label{F-3.1}\\
\Im f(A+i\eps B)&=\sum_{m=1}^l{(-1)^{m-1}\eps^{2m-1}\over(2m-1)!}\,D^{2m-1}f(A;B)
+O(\eps^{2l+1})\quad\mbox{as $\eps\searrow0$}. \label{F-3.2}
\end{align}
Let $k\in\bN$. When $l=2k-1$, formula \eqref{F-3.1} is rewritten as
$$
\Re f(A+i\eps B)=\sum_{m=0}^{2k-2}{(-1)^m\eps^{2m}\over(2m)!}\,D^{2m}f(A;B)
-{\eps^{4k-2}\over(4k-2)!}\,D^{4k-2}f(A;B)+O(\eps^{4k})
$$
as $\eps\searrow0$. If the condition in (1) for matrices is satisfied, then this implies that
$$
D^{4k-2}f(A;B)\ge0
$$
for every $A\in\bM_n^{sa}(a,b)$ and $B\in\bM_n^+$ of any $n\in\bN$. Hence $f$ is operator
$(4k-2)$-tone by condition (a) of Section 2.2 or \cite[Theorem 3.3]{FHR}. When $l=2k$, formula
\eqref{F-3.1} is rewritten as
$$
\Re f(A+i\eps B)=\sum_{m=0}^{2k-1}{(-1)^m\eps^{2m}\over(2m)!}D^{2m}f(A;B)
+{\eps^{4k}\over(4k)!}D^{4k}f(A;B)+O(\eps^{4k+2})
$$
as $\eps\searrow0$. Hence, the condition in (2) for matrices implies that $f$ is operator $4k$-tone.
Similarly, from \eqref{F-3.2} when $l=2k-1$ (resp., when $l=2k$) we see that the condition in (3)
(resp., in (4)) for matrices implies that $f$ is operator $(4k-3)$-tone (resp., operator
$(4k-1)$-tone). So the ``if\," parts of (1)--(4) have been proved.

To prove the ``only if\," part, we need to give some lemmas.

\begin{lemma}\label{L-3.2}
Let $k\in\bN$ and let $P(x)$ be a polynomial with real coefficients and degree $\le k$. Then for every
$A,B\in B(\cH)^{sa}$,
\begin{align*}
\Re P(A+iB)&=\sum_{m=0}^{[k/2]}{(-1)^m\over(2m)!}\,D^{2m}P(A;B), \\
\Im P(A+iB)&=\sum_{m=1}^{[(k+1)/2]}{(-1)^{m-1}\over(2m-1)!}\,D^{2m-1}P(A;B).
\end{align*}
\end{lemma}

\begin{proof}
Since the assertion is obvious when $P(x)$ is a constant, we may assume that $P(x)=x^l$ where
$l\in\{1,\dots,k\}$. Then
$$
P(A+iB)=(A+iB)^l=\sum_{m=0}^li^mF_{l-m,m}(A,B),
$$
where $F_{l-m,m}(A,B)$ denotes the sum of all products of $l-m$ $A$'s and $m$ $B$'s. So we have
\begin{align*}
\Re\bigl\{(A+iB)^l\bigr\}&=\sum_{m=0}^{[l/2]}(-1)^mF_{l-2m,2m}(A,B), \\
\Im\bigl\{(A+iB)^l\bigr\}&=\sum_{m=1}^{[(l+1)/2]}(-1)^{m-1}F_{l-2m+1,2m-1}(A,B).
\end{align*}
On the other hand, it is obvious that
$$
D^mx^l(A;B)={d^m\over dt^m}(A+tB)^l\Big|_{t=0}
=\begin{cases}m!F_{l-m,m}(A,B) & \text{if $0\le m\le l$}, \\
0 & \text{if $m>l$}.
\end{cases}
$$
Hence we have the desired formulas for $\Re\bigl\{(A+iB)^l\bigr\}$ and $\Im\bigl\{(A+iB)^l\bigr\}$.
\end{proof}

\begin{lemma}\label{L-3.3}
Let $k\in\bN$ and $\lambda\in[-1,1]$. Let $A\in B(\cH)^{sa}(-1,1)$ and $B\in B(\cH)^{sa}$,
and set $X:=A+iB$. Then
\begin{align*}
\Re\bigl\{X^{2k}(I-\lambda X)^{-1}\bigr\}
&\le\sum_{m=0}^{k-1}{(-1)^m\over(2m)!}\,D^{2m}\biggl({x^{2k}\over1-\lambda x}\biggr)(A;B)
\quad\mbox{if $k$ is odd}, \\
\Re\bigl\{X^{2k}(I-\lambda X)^{-1}\bigr\}
&\ge\sum_{m=0}^{k-1}{(-1)^m\over(2m)!}\,D^{2m}\biggl({x^{2k}\over1-\lambda x}\biggr)(A;B)
\quad\mbox{if $k$ is even}.
\end{align*}
\end{lemma}

\begin{proof}
First, note that $x^{2k}/(1-\lambda x)$ has the analytic continuation $z^{2k}/(1-\lambda z)$ that is
analytic in $(\bC\setminus\bR)\cup(-1,1)$, so $X^{2k}(I-\lambda X)^{-1}$ is well defined. When
$\lambda=0$, the required inequalities hold with equality by Lemma \ref{L-3.2}. So assume that
$\lambda\ne0$; then we write
$$
{x^{2k}\over1-\lambda x}
={1\over\lambda^{2k}}\cdot{\{1-(1-\lambda x)\}^{2k}\over1-\lambda x}
={1\over\lambda^{2k}}\cdot{1\over1-\lambda x}+P_\lambda(x),
$$
where
$$
P_\lambda(x):={1\over\lambda^{2k}}\sum_{l=1}^{2k}{2k\choose l}(-1)^l(1-\lambda x)^{l-1}.
$$
Hence we have
$$
\Re\bigl\{X^{2k}(I-\lambda X)^{-1}\bigr\}
={1\over\lambda^{2k}}\,\Re\bigl\{(I-\lambda X)^{-1}\bigr\}+\Re P_\lambda(X),
$$
\begin{align*}
&\sum_{m=0}^{k-1}{(-1)^m\over(2m)!}\,
D^{2m}\biggl({x^{2k}\over1-\lambda x}\biggr)(A;B) \\
&\qquad={1\over\lambda^{2k}}\sum_{m=0}^{k-1}{(-1)^m\over(2m)!}\,
D^{2m}\biggl({1\over1-\lambda x}\biggr)(A;B)
+\sum_{m=0}^{k-1}{(-1)^m\over(2m)!}\,D^{2m}P_\lambda(A;B),
\end{align*}
and by Lemma \ref{L-3.2}
$$
\Re P_\lambda(X)=\sum_{m=0}^{k-1}{(-1)^m\over(2m)!}\,D^{2m}P_\lambda(A;B).
$$
Therefore, it suffices to show that
\begin{align*}
\Re\bigl\{(I-\lambda X)^{-1}\bigr\}
&\le\sum_{m=0}^{k-1}{(-1)^m\over(2m)!}\,D^{2m}\biggl({1\over1-\lambda x}\biggr)(A;B)
\quad\mbox{if $k$ is odd}, \\
\Re\bigl\{(I-\lambda X)^{-1}\bigr\}
&\ge\sum_{m=0}^{k-1}{(-1)^m\over(2m)!}\,D^{2m}\biggl({1\over1-\lambda x}\biggr)(A;B)
\quad\mbox{if $k$ is even}.
\end{align*}
Since $I-\lambda A>0$, we define an operator $C_\lambda\in B(\cH)^{sa}$ by
$$
C_\lambda:=\lambda(I-\lambda A)^{-1/2}B(I-\lambda A)^{-1/2}.
$$
We then have
\begin{align}
(I-\lambda X)^{-1}
&=\bigl\{(I-\lambda A)^{1/2}(I-iC_\lambda)(I-\lambda A)^{1/2}\bigr\}^{-1} \nonumber\\
&=(I-\lambda A)^{-1/2}\bigl(I+C_\lambda^2\bigr)^{-1}(I+iC_\lambda)(I-\lambda A)^{-1/2} \label{F-3.3}
\end{align}
so that
$$
\Re\{(I-\lambda X)^{-1}\}=(I-\lambda A)^{-1/2}\bigl(I+C_\lambda^2\bigr)^{-1}(I-\lambda A)^{-1/2}.
$$
On the other hand, for $t\in\bR$ sufficiently near $0$, we have
\begin{align*}
\{I-\lambda(A+tB)\}^{-1}
&=\bigl\{(I-\lambda A)^{1/2}(I-tC_\lambda)(I-\lambda A)^{1/2}\bigr\}^{-1} \\
&=\sum_{m=0}^\infty t^m(I-\lambda A)^{-1/2}C_\lambda^m(I-\lambda A)^{-1/2}
\end{align*}
so that
\begin{equation}\label{F-3.4}
D^m\biggl({1\over1-\lambda x}\biggr)(A;B)
=m!(I-\lambda A)^{-1/2}C_\lambda^m(I-\lambda A)^{-1/2},\qquad m\ge0.
\end{equation}
Therefore,
\begin{align*}
&\Re\bigl\{(I-\lambda X)^{-1}\bigr\}-\sum_{m=0}^{k-1}{(-1)^m\over(2m)!}\,
D^{2m}\biggl({1\over1-\lambda x}\biggr)(A;B) \\
&\quad=(I-\lambda A)^{-1/2}\bigl(I+C_\lambda^2\bigr)^{-1}(I-\lambda A)^{-1/2}
-\sum_{m=0}^{k-1}(-1)^m(I-\lambda A)^{-1/2}C_\lambda^{2m}(I-\lambda A)^{-1/2} \\
&\quad=(I-\lambda A)^{-1/2}
\Bigl\{\bigl(I+C_\lambda^2\bigr)^{-1}-\bigl(I+C_\lambda^2\bigr)^{-1}
\bigl(I-(-1)^kC_\lambda^{2k}\bigr)\Bigr\}(I-\lambda A)^{-1/2} \\
&\quad=(-1)^k(I-\lambda A)^{-1/2}\bigl(I+C_\lambda^2\bigr)^{-1}C_\lambda^{2k}(I-\lambda A)^{-1/2},
\end{align*}
which yields the desired conclusion.
\end{proof}

\begin{lemma}\label{L-3.4}
Let $k\in\bN$ and $\lambda\in[-1,1]$. Let $A\in B(\cH)^{sa}(-1,1)$ and $B\in B(\cH)^+$, and set
$X:=A+iB$. Then
\begin{align*}
\Im\bigl\{X^{2k-1}(I-\lambda X)^{-1}\bigr\}
&\ge\sum_{m=1}^{k-1}{(-1)^{m-1}\over(2m-1)!}\,D^{2m-1}\biggl({x^{2k-1}\over1-\lambda x}\biggr)(A;B)
\quad\mbox{if $k$ is odd}, \\
\Im\bigl\{X^{2k-1}(I-\lambda X)^{-1}\bigr\}
&\le\sum_{m=1}^{k-1}{(-1)^{m-1}\over(2m-1)!}\,D^{2m-1}\biggl({x^{2k-1}\over1-\lambda x}\biggr)(A;B)
\quad\mbox{if $k$ is even}.
\end{align*}
(The right-hand side of the first inequality is $0$ if $k=1$.)
\end{lemma}

\begin{proof}
By Lemma \ref{L-3.2} we may assume that $\lambda\ne0$. Since
$$
{x^{2k-1}\over1-\lambda x}={1\over\lambda^{2k-1}}\cdot{1\over1-\lambda x}
+\tilde P_\lambda(x)
$$
with a polynomial $\tilde P(x)$ of degree $2k-2$, it suffices as in the proof of Lemma \ref{L-3.3}
to prove that
\begin{align*}
\lambda\,\Im\bigl\{(I-\lambda X)^{-1}\bigr\}&\ge\lambda
\sum_{m=1}^{k-1}{(-1)^{m-1}\over(2m-1)!}\,D^{2m-1}\biggl({1\over1-\lambda x}\biggr)(A;B)
\quad\mbox{if $k$ is odd}, \\
\lambda\,\Im\{(I-\lambda X)^{-1}\}&\le\lambda
\sum_{m=1}^{k-1}{(-1)^{m-1}\over(2m-1)!}\,D^{2m-1}\biggl({1\over1-\lambda x}\biggr)(A;B)
\quad\mbox{if $k$ is even}.
\end{align*}
(Here, the inequalities with the factor $1/\lambda^{2k-1}$ are equivalent to those with $\lambda$.)
With the same $C_\lambda$ as in the proof of Lemma \ref{L-3.3}, it follows from \eqref{F-3.3} that
$$
\Im\bigl\{(I-\lambda X)^{-1}\bigr\}
=(I-\lambda A)^{-1/2}\bigl(I+C_\lambda^2\bigr)^{-1}C_\lambda(I-\lambda A)^{-1/2}.
$$
From \eqref{F-3.4} we obtain
\begin{align*}
&\lambda\,\Im\bigl\{(I-\lambda X)^{-1}\bigr\}
-\lambda\sum_{m=1}^{k-1}{(-1)^{m-1}\over(2m-1)!}\,
D^{2m-1}\biggl({1\over1-\lambda x}\biggr)(A;B) \\
&\quad=\lambda(I-\lambda A)^{-1/2}\bigl(I+C_\lambda^2\bigr)^{-1}
C_\lambda(I-\lambda A)^{-1/2} \\
&\quad\qquad-\lambda\sum_{m=1}^{k-1}(-1)^{m-1}(I-\lambda A)^{-1/2}C_\lambda^{2m-1}
(I-\lambda A)^{-1/2} \\
&\quad=\lambda(I-\lambda A)^{-1/2}\Bigl\{\bigl(I+C_\lambda^2\bigr)^{-1}C_\lambda
-\bigl(I+C_\lambda^2\bigr)^{-1}\bigl(C_\lambda-(-1)^{k-1}C_\lambda^{2k-1}\bigr)\Bigr\}
(I-\lambda A)^{-1/2} \\
&\quad=(-1)^{k-1}(I-\lambda A)^{-1/2}\bigl(I+C_\lambda^2\bigr)^{-1}
\bigl(\lambda C_\lambda^{2k-1}\bigr)(I-\lambda A)^{-1/2}.
\end{align*}
Thanks to $B\in B(\cH)^+$ we have
$$
\lambda C_\lambda^{2k-1}
=\lambda^{2k}\bigl\{(I-\lambda A)^{-1/2}B(I-\lambda A)^{-1/2}\bigr\}^{2k-1}\ge0,
$$
and hence the desired conclusion follows.
\end{proof}

Now let us prove the ``only if\," part of Theorem \ref{T-3.1}. Assume that $f$ is operator $l$-tone
on $(a,b)$ where $l\in\bN$, and let $A\in B(\cH)^{sa}(a,b)$ and $B\in B(\cH)^{sa}$. By taking a
finite open interval $(a',b')\subset(a,b)$ such that $\sigma(A)\subset(a',b')$, we may assume that
$(a,b)$ itself is a finite interval. We have a linear transformation $\alpha x+\beta$ with $\alpha>0$
and $\beta\in\bR$ which transforms $(a,b)$ to $(-1,1)$. By replacing $f$ and $A$, $B$ with
$\tilde f(x):=f((x-\beta)/\alpha)$ and $\tilde A:=\alpha A+\beta I$, $\tilde B:=\alpha B$,
respectively, we have
$$
f(A+iB)=\tilde f(\tilde A+i\tilde B),\quad
D^mf(A;B)=D^m\tilde f(\tilde A;\tilde B),\quad m\ge0,
$$
so we end up by assuming that $(a,b)=(-1,1)$. Then by \cite[Theorem 4.1]{FHR}, $f$ admits the
integral expression
$$
f(x)=P(x)+\int_{[-1,1]}{x^l\over1-\lambda x}\,d\mu(\lambda),\qquad x\in(-1,1),
$$
where $P(x)$ is a polynomial (with real coefficients) of degree $<l$ and $\mu$ is a finite positive
measure on $[-1,1]$. For $\lambda\in[-1,1]$ and $x\in(-1,1)$ we set
$$
g_\lambda(x):={x^l\over1-\lambda x}\quad\mbox{and}\quad
g(x):=\int_{[-1,1]}{x^l\over1-\lambda x}\,d\mu(\lambda)
=\int_{[-1,1]}g_\lambda(x)\,d\mu(\lambda).
$$
It is immediate to see that if $H\in B(\cH)^{sa}(-1,1)$, then $g_\lambda(H)=H^l(I-\lambda H)^{-1}$
is continuous in $\lambda\in[-1,1]$ in the operator norm and
\begin{equation}\label{F-7}
g(H)=\int_{[-1,1]}g_\lambda(H)\,d\mu(\lambda),
\end{equation}
where the integral may be considered in the weak sense (in fact, it can be also in the strong sense).
We here prove that
\begin{equation}\label{F-8}
D^mg(A;B)=\int_{[-1,1]}D^mg_\lambda(A;B)\,d\mu(\lambda),\qquad m\ge0.
\end{equation}
Let $r_0:=\max\{|\lambda|:\lambda\in\sigma(A)\}$ ($<1$), $r_1:=(r_0+1)/2$, and
$\Gamma_1:=\{\zeta\in\bC:|\zeta|=r_1\}$. Moreover, choose a $\delta_0>0$ such that
$\delta_0\|B\|\le(1-r_1)/2$, where $\|\cdot\|$ denotes the operator norm. By \eqref{F-2.1} and
\eqref{F-2.2}, for every $\lambda\in[-1,1]$ we have the Taylor expansion of $g_\lambda(A+tB)$ as 
\begin{equation}\label{F-9}
g_\lambda(A+tB)=\sum_{m=0}^\infty{t^m\over m!}\,D^mg_\lambda(A;B),
\qquad t\in\bR,\ |t|<\delta_0,
\end{equation}
where
$$
D^mg_\lambda(A;B)={m!\over2\pi i}\int_{\Gamma_1}g_\lambda(\zeta)
((\zeta I-A)^{-1}B)^m(\zeta I-A)^{-1}\,d\zeta,\qquad m\ge0.
$$
Since
\begin{align*}
&\big\|g_\lambda(\zeta)((\zeta I-A)^{-1}B)^m(\zeta I-A)^{-1}\big\| \\
&\qquad\le|g_\lambda(\zeta)|\cdot\big\|(\zeta I-A)^{-1}\big\|^{m+1}\|B\|^m \\
&\qquad\le{1\over1-r_1}\biggl({1\over r_1-r_0}\biggr)^{m+1}\|B\|^m
=\biggl({1\over1-r_1}\biggr)^{m+2}\|B\|^m,\qquad\zeta\in\Gamma_1,
\end{align*}
it follows that
$$
{\delta_0^m\over m!}\,\big\|\,D^mg_\lambda(A;B)\big\|
\le{r_1\over(1-r_1)^2}\biggl({\delta_0\|B\|\over1-r_1}\biggr)^m
\le{r_1\over(1-r_1)^2}\biggl({1\over2}\biggr)^m,
\qquad\lambda\in[-1,1].
$$
Hence the Taylor expansion in \eqref{F-9} is absolutely convergent in the operator norm uniformly for
$\lambda\in[-1,1]$. So, by \eqref{F-7} (for $H=A+tB$) and the termwise integration of \eqref{F-9} we
obtain
\begin{align*}
g(A+tB)&=\int_{[-1,1]}g_\lambda(A+tB)\,d\mu(\lambda) \\
&=\sum_{m=0}^\infty{t^m\over m!}\int_{[-1,1]}D^mg_\lambda(A;B)\,d\mu(\lambda),
\qquad t\in\bR,\ |t|<\delta_0,
\end{align*}
and the power series in the above right-hand side is convergent in the operator norm for
$|t|<\delta_0$. Therefore, \eqref{F-8} follows.

Furthermore, choose a smooth closed curve $\Gamma$ in $(\bC\setminus\bR)\cup(-1,1)$ surrounding
$\sigma(A+iB)$. Since $\big\|g_\lambda(\zeta)(\zeta I-A-iB)^{-1}\big\|$ is uniformly bounded for
$\zeta\in\Gamma$ and $\lambda\in[-1,1]$, by applying Fubini's theorem we have
\begin{align}
g(A+iB)&={1\over2\pi i}\int_\Gamma g(\zeta)(\zeta I-A-iB)^{-1}\,d\zeta \nonumber\\
&={1\over2\pi i}\int_\Gamma\biggl(\int_{[-1,1]}g_\lambda(\zeta)\,d\mu(\lambda)\biggr)
(\zeta I-A-iB)^{-1}\,d\zeta \nonumber\\
&=\int_{[-1,1]}\biggl({1\over2\pi i}\int_\Gamma
g_\lambda(\zeta)(\zeta I-A-iB)^{-1}\,d\zeta\biggr)\,d\mu(\lambda) \nonumber\\
&=\int_{[-1,1]}g_\lambda(A+iB)\,d\mu(\lambda). \label{F-3.8}
\end{align}

Now we are in a position to complete the proof of Theorem \ref{T-3.1}. Assume that $l=4k-2$, so $f$
is an operator $(4k-2)$-tone function on $(-1,1)$. By \eqref{F-3.8}, Lemmas \ref{L-3.2}, \ref{L-3.3},
and \eqref{F-8} we obtain
\begin{align*}
\Re f(A+iB)&=\Re P(A+iB)+\int_{[-1,1]}\Re g_\lambda(A+iB)\,d\mu(\lambda) \\
&\le\sum_{m=0}^{2k-2}{(-1)^m\over(2m)!}\,D^{2m}P(A;B)
+\int_{[-1,1]}\sum_{m=0}^{2k-2}{(-1)^m\over(2m)!}\,D^{2m}g_\lambda(A;B)\,d\mu(\lambda) \\
&=\sum_{m=0}^{2k-2}{(-1)^m\over(2m)!}\bigl\{D^{2m}P(A;B)+D^{2m}g(A;B)\bigr\} \\
&=\sum_{m=0}^{2k-2}{(-1)^m\over(2m)!}\,D^{2m}f(A;B),
\end{align*}
which completes the proof of (1). The proof of (2) is similar with $l=4k$. On the other hand, when
$l=4k-3$ or $l=4k-1$, by using \eqref{F-3.8}, Lemmas \ref{L-3.2}, \ref{L-3.4}, and \eqref{F-8} we can
similarly prove (3) or (4) under an additional assumption $B\in B(\cH)^+$.
\end{proof}

\begin{remark}\label{R-3.5}\rm
In \cite[Definition 1.4]{FHR} we also introduced the notion of {\it matrix $k$-tone functions of
order $n$} for each fixed $n\in\bN$. As shown in \cite[Proposition 2.6]{FHR}, a function on $(a,b)$
is matrix $k$-tone of order $n$ if and only if condition (a) in Section 2.3 holds with $n$ fixed.
Thus, as seen from the proof of the ``if\," part of Theorem \ref{T-3.1}, if the inequality of (1)
(or (2)--(4), respectively) holds for every $A\in\bM_n^{sa}(a,b)$ and every $B\in\bM_n^+$ of a fixed
$n\in\bN$, then $f$ is matrix $(4k-2)$-tone (or matrix $4k$, $(4k-3)$, $(4k-1)$-tone, respectively)
of order $n$. However, the converse direction seems subtle even when $n=1$. Consider simple monotone
functions $f(x):=x^p$ on $(0,\infty)$, where $p>0$. Inequality in (3) of Theorem \ref{T-3.1} when
$k=n=1$ means that $\Im(a+ib)^p\ge0$ for every $a,b>0$. This holds if and only if $(0<)\ p\le2$, but
$x^p$ is matrix $1$-tone of order $1$ for any $p>0$ since it means monotonicity in the numerical
sense. Inequality in (1) of Theorem \ref{T-3.1} when $k=n=1$ means that $\Re(a+ib)^p\le a^p$ for
every $a,b>0$. This holds if and only if $1\le p\le3$, but $x^p$ is matrix $2$-tone of order $1$ for
any $p\ge2$ since it means convexity in the usual sense.
\end{remark}

\section{Operator monotone and operator convex functions}

In this section our consideration is specialized to operator monotone (also operator monotone
decreasing) functions and operator convex functions, for which we give further characterizations
via analytic functional calculus.

First, we restate (1) of Theorem \ref{T-3.1} in the particular case $k=1$ as a corollary for
convenience of later references.

\begin{cor}\label{C-4.1}
Let $f$ be as in Theorem \ref{T-3.1}. Then the following conditions are equivalent:
\begin{itemize}
\item[\rm(i)] $f$ is operator convex on $(a,b)$;
\item[\rm(ii)] $\Re f(A+iB)\le f(A)$ for every $A\in B(\cH)^{sa}(a,b)$ and every
$B\in B(\cH)^{sa}$;
\item[\rm(iii)] $\Re f(A+iB)\le f(A)$ for every $A\in\bM_n^{sa}(a,b)$ and $B\in\bM_n^+$ of
any $n\in\bN$.
\end{itemize}
\end{cor}

The next theorem is considered as the operator-valued version of L\"owner's theorem. Indeed, condition
(iv) is well-known L\"owner's characterization of operator monotone functions in terms of analytic
continuation as Pick functions, and (ii) is its operator-valued version.

\begin{thm}\label{T-4.2}
Let $f$ be as in Theorem \ref{T-3.1}. Then the following conditions are equivalent:
\begin{itemize}
\item[\rm(i)] $f$ is operator monotone on $(a,b)$;
\item[\rm(ii)] $\Im f(X)\ge0$ for every $X\in B(\cH)$ such that $\Im X>0$;
\item[\rm(iii)] $\Im f(A+iB)\ge0$ for every $A,B\in\bM_n^{sa}(a,b)$ with $B>0$ of any $n\in\bN$;
\item[\rm(iv)] $\Im f(z)\ge0$ for every $z\in\bC$ with $\Im z>0$.
\end{itemize}

Moreover, if $f$ is a non-constant operator monotone function on $(a,b)$, then $\Im f(X)>0$ for
every $X\in B(\cH)$ with $\Im X>0$.
\end{thm}

\begin{proof}
The implications (ii) $\Rightarrow$ (iii) and (ii) $\Rightarrow$ (iv) are trivial,
(i) $\Leftrightarrow$ (iii) is the particular $k=1$ case of (3) of Theorem \ref{T-3.1}, and
(i) $\Leftrightarrow$ (iv) is L\"owner's theorem. So we may prove that (i) $\Rightarrow$ (ii), which
is not contained in Theorem \ref{T-3.1}. Note by Lemma \ref{L-2.1} that $f(X)$ in (ii) is well
defined.

As in the proof of Theorem \ref{T-3.1} we may assume that $(a,b)=(-1,1)$. So assume that $f$ is
operator monotone on $(-1,1)$; then the well-known integral expression of $f$ is
$$
f(x)=f(0)+\int_{[-1,1]}{x\over1-\lambda x}\,d\mu(\lambda),\qquad x\in(-1,1),
$$
where $\mu$ is a finite positive measure on $[-1,1]$. Assume that $X=A+iB$ with $A,B\in B(\cH)^{sa}$
and $B>0$. For every $\lambda\in[-1,1]$, letting $C_\lambda:=B^{-1/2}(I-\lambda A)B^{-1/2}$ we have
$$
I-\lambda X=B^{1/2}(C_\lambda-i\lambda I)B^{1/2}
$$
so that
\begin{equation}\label{F-4.1}
(I-\lambda X)^{-1}
=B^{-1/2}\bigl(C_\lambda^2+\lambda^2I\bigr)^{-1}(C_\lambda+i\lambda I)B^{-1/2}.
\end{equation}
Since $C_0=B^{-1}>0$, it is clear that $\bigl(C_\lambda^2+\lambda^2I\bigr)^{-1}(C_\lambda+i\lambda I)$
is continuous in the operator norm in $\lambda\in[-1,1]$. Hence $\big\|(I-\lambda X)^{-1}\big\|$ is
uniformly bounded for $\lambda\in[-1,1]$ so that the analytic functional calculus $f(X)$ is given as
$$
f(X)=f(0)I+\int_{[-1,1]}X(I-\lambda X)^{-1}\,d\mu(\lambda).
$$
Therefore,
$$
\Im f(X)=\int_{[-1,1]}\Im\bigl\{X(I-\lambda X)^{-1}\bigr\}\,d\mu(\lambda).
$$
So it suffices to show that
$$
\Im\bigl\{X(I-\lambda X)^{-1}\bigr\}>0,\qquad\lambda\in[-1,1].
$$
This is obvious for $\lambda=0$. For $\lambda\in[-1,1]$, $\lambda\ne0$, since
$$
X(I-\lambda X)^{-1}=-{1\over\lambda}\,I+{1\over\lambda}(I-\lambda X)^{-1},
$$
we obtain, thanks to \eqref{F-4.1},
$$
\Im\bigl\{X(I-\lambda X)^{-1}\bigr\}={1\over\lambda}\,\Im\bigl\{(I-\lambda X)^{-1}\bigr\}
=B^{-1/2}\bigl(C_\lambda^2+\lambda^2I\bigr)^{-1}B^{-1/2}>0.
$$
Moreover, the above proof shows that if $\Im f(X)>0$ is not satisfied, then $\mu=0$ and so
$f$ is a constant function, implying the second assertion of the theorem.
\end{proof}

In the next theorem we have some further characterization results when $(a,b)=(0,\infty)$.

\begin{thm}\label{T-4.3}
Let $f$ be a continuous real function on $(0,\infty)$ analytically continued to $\bC^+$.
Then:
\begin{itemize}
\item[\rm(1)] The following conditions are equivalent:
\begin{itemize}
\item[\rm(i)] $f$ is non-negative and operator monotone on $(0,\infty)$;
\item[\rm(ii)] $f$ admits the integral expression
\begin{equation}\label{F-4.2}
f(x)=\alpha+\beta x+\int_{(0,\infty)}{x\over x+\lambda}\,d\mu(\lambda),
\qquad x\in(0,\infty),
\end{equation}
where $\alpha,\beta\ge0$ and $\mu$ is a positive measure on $(0,\infty)$ such that
$\int_{(0,\infty)}(1+\lambda)^{-1}\,d\mu(\lambda)\allowbreak<+\infty$;
\item[\rm(iii)] $0\le f(\Re X)\le\Re f(X)$ for every $X\in B(\cH)$ with $\Re X>0$.
\end{itemize}

\item[\rm(2)] The following conditions are equivalent:
\begin{itemize}
\item[\rm(i)] there are a $\beta\ge0$ and a non-negative operator monotone decreasing function $g$
on $(0,\infty)$ such that $f(x)=\beta x+g(x)$ for all $x\in(0,\infty)$;
\item[\rm(ii)] $f$ admits the integral expression
$$
f(x)=\alpha+\beta x+\int_{[0,\infty)}{1\over x+\lambda}\,d\mu(\lambda),
\qquad x\in(0,\infty),
$$
where $\alpha,\beta\ge0$ and $\mu$ is a positive measure on $[0,\infty)$ such that
$\int_{[0,\infty)}(1+\lambda)^{-1}\,d\mu(\lambda)\allowbreak<+\infty$;
\item[\rm(iii)] $0\le\Re f(X)\le f(\Re X)$ for every $X\in B(\cH)$ with $\Re X>0$;
\item[\rm(iv)] $\Re f(X)\le f(\Re X)$ for every $X\in B(\cH)$ with $\Re X>0$, and $\Re f(z)\ge0$ for
every $z\in\bC$ with $\Re z>0$.
\end{itemize}

Moreover, if the above equivalent conditions are satisfied, then $\Re f(X)>0$ for every $X\in B(\cH)$
with $\Re X>0$ unless $f$ is identically zero.

\item[\rm(3)] Assume that $f$ is non-negative on $(0,\infty)$ and it is not identically zero. Then
the following conditions are equivalent:
\begin{itemize}
\item[\rm(i)] $f$ is operator monotone decreasing on $(0,\infty)$;
\item[\rm(ii)] $f$ admits the integral expression
$$
f(x)=\alpha+\int_{[0,\infty)}{1\over x+\lambda}\,d\mu(\lambda),\qquad x\in(0,\infty),
$$
where $\alpha\ge0$ and $\mu$ is a positive measure on $[0,\infty)$ such that
$\int_{[0,\infty)}(1+\lambda)^{-1}\,d\mu(\lambda)<+\infty$;
\item[\rm(iii)] $\Im f(X)\ge0$ for every $X\in B(\cH)$ with $\Im X<0$;
\item[\rm(iv)] $0<\Re f(X)\le f(\Re X)$ and $\Re(\log f(X))\le\log f(\Re X)$ for every $X\in B(\cH)$
with $\Re X>0$.
\end{itemize}
\end{itemize}
\end{thm}

\begin{proof}
(1)\enspace
(i) $\Leftrightarrow$ (ii) is well-known (see, e.g., \cite[\S V.4]{Bh1}, \cite[\S2,7]{Hi}); we state
it just for the sake of completeness. The first inequality of (iii) is of course equivalent to the
non-negativity of $f$. Hence (i) $\Leftrightarrow$ (iii) follows from Corollary \ref{C-4.1} since
operator monotonicity and operator concavity are equivalent for a function on $(0,\infty)$.

(2)\enspace
(i) $\Leftrightarrow$ (ii) here is also well-known (see \cite{Ha}, \cite[Theorems 3.1]{AH}). Assume
(ii). Since $f$ is operator convex, Corollary \ref{C-4.1} implies that the second inequality of (iii)
holds for every $X\in B(\cH)$ with $\Re X>0$. Moreover, write $X=A+iB$ with $A,B\in B(\cH)^{sa}$ and
$A>0$. We then have
$$
\Re f(X)=\alpha I+\beta A+\int_{[0,\infty)}\Re\bigl\{(X+\lambda I)^{-1}\bigr\}\,d\mu(\lambda).
$$
For every $\lambda\ge0$ we have
$$
(X+\lambda I)^{-1}
=(A+\lambda I)^{-1/2}\bigl(I+C_\lambda^2\bigr)^{-1}(I-iC_\lambda)(A+\lambda I)^{-1/2},
$$
where $C_\lambda:=(A+\lambda I)^{-1/2}B(A+\lambda I)^{-1/2}$, so that
$$
\Re\bigl\{(X+\lambda I)^{-1}\bigr\}
=(A+\lambda I)^{-1/2}\bigl(I+C_\lambda^2\bigr)^{-1}(A+\lambda I)^{-1/2}>0.
$$
Therefore, the first inequality of (iii) follows. Moreover, if $\Re f(X)>0$ does not hold for some
$X$ as above, then $\alpha=\beta=0$ and $\mu=0$, i.e., $f$ is identically zero.

(iii) $\Rightarrow$ (iv) is trivial. Finally, assume (iv). By Corollary \ref{C-4.1}, $f$ is
operator convex so that by \cite[Theorem 5.1]{FHR} it admits the integral expression
$$
f(x)=f(1)+f'(1)(x-1)+\gamma(x-1)^2+\int_{[0,\infty)}{(x-1)^2\over x+\lambda}
\,d\mu(\lambda),\qquad x\in(0,\infty),
$$
where $\gamma\ge0$ and $\mu$ is a positive measure on $[0,\infty)$ such that
$\int_{[0,\infty)}(1+\lambda)^{-1}\,d\mu(\lambda)<+\infty$. For $z=1+ib\in\bC$ with any
$b\in\bR$, (iv) gives
$$
0\le\Re f(z)=f(1)-\gamma b^2-\int_{[0,\infty)}{(1+\lambda)b^2\over(1+\lambda)^2+b^2}
\,d\mu(\lambda).
$$
This implies that $\gamma=0$. Furthermore, since
$(1+\lambda)b^2/\{(1+\lambda)^2+b^2\}\nearrow1+\lambda$ as $b^2\nearrow\infty$, by the
monotone convergence theorem we have
$$
\int_{[0,\infty)}(1+\lambda)\,d\mu(\lambda)\le f(1)<+\infty.
$$
Since
$$
{(x-1)^2\over x+\lambda}=x-(2+\lambda)+{(1+\lambda)^2\over x+\lambda},
$$
we can write
\begin{align*}
f(x)&=f(1)+f'(1)(x-1) \\
&\qquad+\mu([0,\infty))x-\int_{[0,\infty)}(2+\lambda)\,d\mu(\lambda)
+\int_{[0,\infty)}{(1+\lambda)^2\over x+\lambda}\,d\mu(\lambda) \\
&=\alpha+\beta x+\int_{[0,\infty)}{1+\lambda\over x+\lambda}\,d\nu(\lambda),
\end{align*}
where $\alpha,\beta\in\bR$ and $d\nu(\lambda):=(1+\lambda)\,d\mu(\lambda)$ is a finite positive
measure on $[0,\infty)$. For every $z=a+ib$ with $a>0$ and $b\in\bR$ we have
$$
0\le\Re f(z)=\alpha+\beta a
+\int_{[0,\infty)}{(a+\lambda)(1+\lambda)\over(a+\lambda)^2+b^2}\,d\nu(\lambda).
$$
By the bounded convergence theorem the above integral term converges to $0$ as $b^2\nearrow\infty$,
so $\alpha+\beta a\ge0$ for all $a>0$. This gives $\alpha,\beta\ge0$. Hence (iv) $\Rightarrow$ (ii)
is proved.

(3)\enspace
(i) $\Leftrightarrow$ (ii) is well-known as in the proof of (2). (i) $\Leftrightarrow$ (iii)
immediately follows from Theorem \ref{T-4.2} when applied to $f(x^{-1})$ on $(0,\infty)$ since
$\Im X<0$ is equivalent to that $X$ is invertible and $\Im X^{-1}>0$. It is known
\cite[Theorem 3.1]{AH} that a continuous function $f>0$ on $(0,\infty)$ is operator monotone
decreasing if and only if both $f$ and $\log f$ are operator convex. Moreover, since $f$ is not
identically zero, (i) implies from (2) that $0<\Re f(X)\le f(\Re X)$ for every $X\in B(\cH)$ with
$\Re X>0$. Hence (i) $\Leftrightarrow$ (iv) follows from Corollary \ref{C-4.1}. (Here, note that
$\log f(X)$ is well defined for $\Re X>0$ since $\Re f(X)>0$ implies that the spectrum $\sigma(f(X))$
is in $\{z\in\bC:\Re z>0\}$.)
\end{proof}

According to \cite[Proposition 3.9]{FHR}, if $f$ is an operator $k_0$-tone function on $(a,b)$ for
some $k_0\in\bN$, then it is operator $(k_0+2k)$-tone on $(a,b)$ for any $k\in\bN$. Hence, if $f$ is
operator $2k_0$-tone on $(a,b)$ for some $k_0\in\bN$, then the inequalities in (1) and in (2) of
Theorem \ref{T-3.1} hold for every $k\ge[k_0/2]+1$ and for every $k\ge[(k_0+1)/2]$, respectively.
Also, if $f$ is operator $(2k_0-1)$-tone on $(a,b)$ for some $k_0\in\bN$, then the inequalities in
(3) and in (4) of Theorem \ref{T-3.1} hold for every $k\ge[k_0/2]+1$ and for every $k\ge[(k_0+1)/2]$,
respectively. Moreover, according to \cite[Proposition 5.2]{FHR}, if $f$ is operator $k_0$-tone on
$(0,\infty)$ for some $k_0\in\bN$, then $(-1)^kf$ is operator $(k_0+k)$-tone on $(0,\infty)$ for any
$k\in\bN$. Therefore, Theorem 3.1 yields the following:

\begin{cor}\label{C-4.4}
Assume that $f$ is operator monotone on $(0,\infty)$. Then for every $k\in\bN$,
$$
\sum_{m=0}^{2k-2}{(-1)^m\over(2m)!}D^{2m}f(A;B)\le\Re f(A+iB)
\le\sum_{m=0}^{2k-1}{(-1)^m\over(2m)!}D^{2m}f(A;B)
$$
for every $A,B\in B(\cH)^{sa}$ with $A>0$, and
$$
\sum_{m=1}^{2k-2}{(-1)^{m-1}\over(2m-1)!}D^{2m-1}f(A;B)\le\Im f(A+iB)
\le\sum_{m=1}^{2k-1}{(-1)^{m-1}\over(2m-1)!}D^{2m-1}f(A;B)
$$
for every $A,B\in B(\cH)^{sa}$ with $A>0$ and $B\ge0$.
\end{cor}

\begin{cor}\label{C-4.5}
Assume that $f$ is operator convex on $(0,\infty)$. Then for every $k\in\bN$,
$$
\sum_{m=0}^{2k-1}{(-1)^m\over(2m)!}D^{2m}f(A;B)\le\Re f(A+iB)
\le\sum_{m=0}^{2k-2}{(-1)^m\over(2m)!}D^{2m}f(A;B)
$$
for every $A,B\in B(\cH)^{sa}$ with $A>0$, and
$$
\sum_{m=1}^{2k-1}{(-1)^{m-1}\over(2m-1)!}D^{2m-1}f(A;B)\le\Im f(A+iB)
\le\sum_{m=1}^{2k}{(-1)^{m-1}\over(2m-1)!}D^{2m-1}f(A;B)
$$
for every $A,B\in B(\cH)^{sa}$ with $A>0$ and $B\ge0$.
\end{cor}

In the rest of this section let us strengthen (ii) of Theorem \ref{T-4.2} in the case
$(a,b)=(0,\infty)$. Assume that $f$ is a non-negative operator monotone function on $(0,\infty)$, and
let $f(z)$ be the analytic continuation (a Pick function) to $\bC\setminus(-\infty,0]$. It is rather
well-known (and easily verified by using the integral expression \eqref{F-4.2} of $f$) that if
$z\in\bC\setminus\{0\}$ and $0<\arg z<p\pi$ where $0<p\le1$, then $f(z)\in[0,\infty)$ or
$0<\arg f(z)<p\pi$, that is, the argument of $f(z)$ does not exceed that of $z$ for every $z\in\bC$
with $\Im z>0$. The theorem below is the operator-valued version of this result. For each $p\in(0,1]$
define
\begin{align*}
V_{p\pi}&:=\bigl\{X\in B(\cH):\Im X>0,\,\Im(e^{-ip\pi}X)<0\bigr\}, \\
V_{-p\pi}&:=\bigl\{X\in B(\cH):\Im X<0,\,\Im(e^{ip\pi}X)>0\bigr\},
\end{align*}
where $A<0$ means that $-A>0$. Obviously, $V_{p\pi}$ and $V_{-p\pi}$ are the operator
counterparts of $\{z\in\bC\setminus\{0\}:0<\arg z<p\pi\}$ and
$\{z\in\bC\setminus\{0\}:0>\arg z>-p\pi\}$, respectively.

\begin{thm}\label{T-4.6}
Let $0<p\le1$.
\begin{itemize}
\item[\rm(1)] Assume that $f$ is a non-negative and operator monotone function on $(0,\infty)$. If
$X\in V_{p\pi}$ (resp., $X\in V_{-p\pi}$), then $f(X)\in V_{p\pi}$ (resp., $f(X)\in V_{-p\pi}$), or
else $f(X)=\alpha I$ with $\alpha\ge0$. Moreover, $f(X)\in V_{p\pi}$ (resp., $f(X)\in V_{-p\pi}$) for
every $X\in V_{p\pi}$ (resp., $X\in V_{-p\pi}$) unless $f$ is a constant function.

\item[\rm(2)] Assume that $f$ is a non-negative and operator monotone decreasing function on
$(0,\infty)$. If $X\in V_{p\pi}$ (resp., $X\in V_{-p\pi}$), then $f(X)\in V_{-p\pi}$ (resp.,
$f(X)\in V_{p\pi}$), or else $f(X)=\alpha I$ with $\alpha\ge0$. Moreover, $f(X)\in V_{-p\pi}$ (resp.,
$f(X)\in V_{p\pi}$) for every $X\in V_{p\pi}$ (resp., $X\in V_{-p\pi}$) unless $f$ is a constant
function.
\end{itemize}
\end{thm}

\begin{proof}
(1)\enspace
By assumption $f$ admits the integral expression \eqref{F-4.2} as in (ii) of Theorem \ref{F-4.3}\,(1).
Assume that $X\in V_{p\pi}$, so $\Im X>0$ and $\Im(e^{-ip\pi}X)<0$. Theorem \ref{T-4.2} says that
$\Im f(X)\ge0$ and that if $\Im f(X)>0$ is not satisfied, then $f$ is a constant function. As before
we have
$$
f(X)=\alpha I+\beta X+\int_{(0,\infty)}X(X+\lambda I)^{-1}\,d\mu(\lambda).
$$
Let us prove that
\begin{equation}\label{F-4.3}
\Im\bigl\{e^{-ip\pi}X(X+\lambda I)^{-1}\bigr\}<0,\qquad\lambda\in(0,\infty).
\end{equation}
Write $X=A+iB$ with $A,B\in B(\cH)^{sa}$. Letting $C_\lambda:=B^{-1/2}(A+\lambda I)B^{-1/2}$ we
compute
\begin{align*}
X(X+\lambda I)^{-1}&=I-\lambda(X+\lambda I)^{-1} \\
&=I-\lambda\bigl\{B^{1/2}(C_\lambda+iI)B^{1/2}\bigr\}^{-1} \\
&=I-\lambda B^{-1/2}\bigl(C_\lambda^2+I\bigr)^{-1}(C_\lambda-iI)B^{-1/2} \\
&=I-\lambda B^{-1/2}\bigl(C_\lambda^2+I\bigr)^{-1}C_\lambda B^{-1/2}
+i\lambda B^{-1/2}\bigl(C_\lambda^2+I\bigr)^{-1}B^{-1/2}
\end{align*}
so that
\begin{align}
&\Im\bigl\{e^{-ip\pi}X(X+\lambda I)^{-1}\bigr\} \nonumber\\
&\quad=\lambda\cos p\pi\cdot B^{-1/2}\bigl(C_\lambda^2+I\bigr)^{-1}B^{-1/2}
-\sin p\pi\cdot\Bigl\{I-\lambda B^{-1/2}\bigl(C_\lambda^2+I\bigr)^{-1}C_\lambda B^{-1/2}\Bigr\}
\nonumber\\
&\quad=B^{-1/2}\bigl(C_\lambda^2+I\bigr)^{-1/2} \nonumber\\
&\qquad\quad\times\Bigl\{\lambda\cos p\pi\cdot I+\lambda\sin p\pi\cdot C_\lambda
-\sin p\pi\cdot\bigl(C_\lambda^2+I\bigr)^{1/2}B\bigl(C_\lambda^2+I\bigr)^{1/2}\Bigr\} \nonumber\\
&\qquad\quad\times\bigl(C_\lambda^2+I\bigr)^{-1/2}B^{-1/2}. \label{F-4.4}
\end{align}
The assumption $\Im(e^{-ip\pi}X)<0$ means that $\cos p\pi\cdot B-\sin p\pi\cdot A<0$ so that
$$
\cos p\pi\cdot I<\sin p\pi\cdot B^{-1/2}AB^{-1/2}
=\sin p\pi\cdot\bigl(C_\lambda-\lambda B^{-1}\bigr).
$$
Thanks to $\lambda>0$ and $\sin p\pi\ge0$ we hence have
\begin{align}
&\lambda\cos p\pi\cdot I+\lambda\sin p\pi\cdot C_\lambda
-\sin p\pi\cdot\bigl(C_\lambda^2+I\bigr)^{1/2}B\bigl(C_\lambda^2+I\bigr)^{1/2} \nonumber\\
&\quad<\lambda\sin p\pi\cdot\bigl(C_\lambda-\lambda B^{-1}\bigr)
+\lambda\sin p\pi\cdot C_\lambda
-\sin p\pi\cdot\bigl(C_\lambda^2+I\bigr)^{1/2}B\bigl(C_\lambda^2+I\bigr)^{1/2} \nonumber\\
&\quad=\sin p\pi\cdot B^{-1/2}
\Bigl\{-\lambda^2I+2\lambda B^{1/2}C_\lambda B^{1/2}
-\bigl(B^{1/2}\bigl(C_\lambda^2+I\bigr)^{1/2}B^{1/2}\bigr)^2\Bigr\}B^{-1/2}. \label{F-4.5}
\end{align}
Furthermore, since $C_\lambda\le\bigl(C_\lambda^2+I\bigr)^{1/2}$, we have
\begin{align}
&-\lambda^2I+2\lambda B^{1/2}C_\lambda B^{1/2}
-\bigl(B^{1/2}\bigl(C_\lambda^2+I\bigr)^{1/2}B^{1/2}\bigr)^2 \nonumber\\
&\qquad\le-\lambda^2 I+2\lambda B^{1/2}\bigl(C_\lambda^2+I\bigr)^{1/2}B^{1/2}
-\bigl(B^{1/2}\bigl(C_\lambda^2+I\bigr)^{1/2}B^{1/2}\bigr)^2 \nonumber\\
&\qquad=-\Bigl\{\lambda I-B^{1/2}\bigl(C_\lambda^2+I\bigr)^{1/2}B^{1/2}\Bigr\}^2\le0. \label{F-4.6}
\end{align}
Combining \eqref{F-4.4}--\eqref{F-4.6} altogether we arrive at \eqref{F-4.3}.

Since $\Im(e^{-ip\pi}\alpha)\le0$ and $\Im(e^{-ip\pi}\beta X)\le0$ ($<0$ if $\beta>0$), we see by
\eqref{F-4.3} that $\Im\{e^{-p\pi}f(X)\}\le0$. Moreover, if $\Im\{e^{-ip\pi}f(X)\}<0$ is not
satisfied, then $\beta=0$ and $\mu=0$; hence $f$ is a constant $\alpha$ so that $f(X)=\alpha I$.
This implies that $f(X)\in V_{p\pi}$, or else $f$ is a constant. The assertion of (1) with $V_{-p\pi}$
in place of $V_{p\pi}$ is immediately seen since we have $f(X)=f(X^*)^*$ by reflection principle.

(2)\enspace
Note that $f$ is operator monotone on $(0,\infty)$ if and only if $f(x^{-1})$ is operator monotone
decreasing on $(0,\infty)$. Also, if $X\in V_{p\pi}$ (resp., $X\in V_{-p\pi}$), then $X$ is invertible
and $X^{-1}\in V_{-p\pi}$ (resp., $X^{-1}\in V_{p\pi}$). Hence (2) is immediately seen from (1).
\end{proof}

\begin{remark}\label{R-4.8}\rm
Assume that $f$ is a non-negative and non-decreasing operator convex function on $(0,\infty)$. One
has the integral expression
$$
f(x)=\alpha+\beta t+\gamma t^2+\int_{(0,\infty)}{x^2\over x+\lambda}\,d\mu(\lambda),
\qquad x\in(0,\infty),
$$
where $\alpha,\beta,\gamma\ge0$ and $\mu$ is a positive measure on $(0,\infty)$ such that
$\int_{(0,\infty)}(1+\lambda)^{-1}\,d\mu(\lambda)\allowbreak<+\infty$ (see \cite[\S\S2.7--2.8]{Hi}).
From this expression it is easily verified that if $z\in\bC\setminus\{0\}$ and $0<\arg z<p\pi$ where
$0<p\le1/2$, then $f(z)\in[0,\infty)$ or $0<\arg f(z)<2p\pi$, that is, the argument of $f(z)$ does
not exceed 2 times that of $z$ for every $z$ in the first quadrant of $\bC$. The operator-valued
version of this like Theorem \ref{T-4.6} holds under the assumption that $X$ is normal, but it is
not true in general. Indeed, if this holds for $f(x):=x^2$, then we must have
$\Im\{(A+iB)^2\}=AB+BA\ge0$ for $A,B\in B(\cH)^{sa}$ with $A,B>0$. However, this is not valid in
general.
\end{remark}

\end{document}